\documentclass{actams-author}

\numberwithin{equation}{section} 
\theoremstyle{plain}
\newtheorem{thm}{Theorem}[section]
\newtheorem{cor}[thm]{Corollary}
\newtheorem{lem}[thm]{Lemma}
\newtheorem{prop}[thm]{Proposition}

\theoremstyle{definition}
\newtheorem{defn}{Definition}[section]

\newtheorem{rmk}{Remark}[section]

\newcommand{\eps}{\varepsilon}

\DeclareMathOperator*{\esssup}{esssup}
\DeclareMathOperator*{\essinf}{essinf}

\newcommand{\cO}{\mathcal{O}}

\newcommand{\cL}{\mathcal{L}}
\newcommand{\cT}{\mathcal{T}}

\newcommand{\cB}{\mathcal{B}}
\newcommand{\cA}{\mathcal{A}}
\newcommand{\cS}{\mathcal{S}}

\newcommand{\cG}{\mathcal{G}}

\newcommand{\cU}{\mathcal{U}}

\newcommand{\bH}{\mathbb{H}}
\newcommand{\bP}{\mathbb{P}}
\newcommand{\bR}{\mathbb{R}}
\newcommand{\bN}{\mathbb{N}}

\newcommand{\sF}{\mathscr{F}}
\newcommand{\sP}{\mathscr{P}}

\begin{document}
\ID{E13-xxx} 
 \DATE{Final, 2018-07-12} 
 \PageNum{1}
 \Volume{201x}{}{3x}{x} 
 \EditorNote{$^*$Received July 18, 2018; revised January 5, 2019. Jinniao Qiu is partially supported by the National Science and Engineering Research Council of Canada (NSERC) and by the start-up funds from the University of Calgary. The support of the NSERC grant of Professor Robert Elliott for Wenning Wei is gratefully acknowledged.} 

\abovedisplayskip 6pt plus 2pt minus 2pt \belowdisplayskip 6pt
plus 2pt minus 2pt
\def\vsp{\vspace{1mm}}
\def\th#1{\vspace{1mm}\noindent{\bf #1}\quad}
\def\proof{\vspace{1mm}\indent{\bf Proof}\quad}
\def\no{\nonumber}
\newenvironment{prof}[1][Proof]{\indent\textbf{#1}\quad }
{\hfill $\Box$\vspace{0.7mm}}
\def\q{\quad} \def\qq{\qquad}
\allowdisplaybreaks[4]


\AuthorMark{Qiu\,\&  Wei}                             

\TitleMark{\uppercase{Viscosity solutions of Stochastic HJ equations}}  

\title{\uppercase{ Uniqueness of Viscosity Solutions of Stochastic Hamilton-Jacobi Equations}        
\footnote{ }}                 
\author{\sl{Jinniao \uppercase{Qiu}}\footnotemark[2]}    
   { Department of Mathematics \& Statistics, University of Calgary, 2500 University Drive NW, Calgary, AB T2N 1N4, Canada.\\
    E-mail\,$:$ jinniao.qiu@ucalgary.ca}

\author{\sl{Wenning \uppercase{Wei}}\footnotemark[2]}    
{Department of Mathematics \& Statistics, University of Calgary, 2500 University Drive NW, Calgary, AB T2N 1N4, Canada.\\
    E-mail\,$:$ wenning.wei@ucalgary.ca}  
\maketitle%

\Abstract{This paper is devoted to the study of fully nonlinear stochastic Hamilton-Jacobi (HJ) equations for the optimal stochastic control problem of ordinary differential equations with random coefficients.  Under the standard Lipschitz continuity assumptions on the coefficients, the value function is proved to be the unique viscosity solution of the associated stochastic HJ equation.}      


\MRSubClass{49L20, 49L25, 93E20, 35D40, 60H15}      

\section{Introduction}
Let $(\Omega,\sF,\{\sF_t\}_{t\geq0},\bP)$ be a complete filtered probability space with the filtration satisfying the usual conditions and generated by an $m$-dimensional Wiener process $W=\{W_t:t\in[0,\infty)\}$ together with all the
$\bP$-null sets in $\sF$. The predictable $\sigma$-algebra  on $\Omega\times[0,T]$ associated with $\{\sF_t\}_{t\geq0}$ is denoted by  $\sP$.  

This paper is devoted to the uniqueness of viscosity solution to the following stochastic Hamilton-Jacobi (HJ) equation:
\begin{equation}\label{SHJB}
  \left\{\begin{array}{l}
  \begin{split}
  -du(t,x)=\,& 
 \mathbb{H}(t,x,Du) 
 \,dt -\psi(t,x)\, dW_{t}, \quad
                     (t,x)\in Q:=[0,T)\times\bR^d;\\
    u(T,x)=\, &G(x), \quad x\in\bR^d,
    \end{split}
  \end{array}\right.
\end{equation}
with 
\begin{align*}
\mathbb{H}(t,x,p)
= \essinf_{v\in U} \bigg\{
       \beta'(t,x,v)p +f(t,x,v)
                \bigg\},\quad \text{for } p\in \bR^d,  
\end{align*}
  where $T\in (0,\infty)$ is a fixed deterministic terminal time, $U\subset \bR^n$ is a nonempty compact set and both the random fields $u(t,x)$ and $\psi(t,x)$ are  unknown.

Stochastic HJ equations like \eqref{SHJB} arise naturally from optimal stochastic control problems of the following form:
\begin{align}
\inf_{\theta\in\cU}E\left[\int_0^T\!\! f(s,X_s,\theta_s)\,ds +G(X_T) \right] \label{Control-probm}
\end{align}
subject to
\begin{equation}\label{state-proces-contrl}
\left\{
\begin{split}
&dX_t=\beta(t,X_t,\theta_t)dt. ,\,\,
\,t\in[0,T]; \\
& X_0=x,
\end{split}
\right.
\end{equation}
 where $\cU$ is the set of all the $U$-valued and $\sF_t$-adapted processes and the coefficients $\beta,f$ and $G$ depend not only on time, space and control but also \textit{explicitly} on $\omega \in\Omega$ (see assumption $(\cA1)$).   The state process $(X_t)_{t\in[0,T]}$  is governed by the {\sl control} $\theta\in\cU$, and to indicate the dependence of the state process on the control $\theta$, the initial time $r$ and initial state $x\in \mathbb{R}^d$, we also write $X^{r,x;\theta}_t$ for $0\leq r\leq t\leq T$. Following the dynamic programming method, we may define the dynamic cost functional
 \begin{align}
J(t,x;\theta)=E_{\sF_t}\left[\int_t^T\!\! f(s,X^{t,x;\theta}_s,\theta_s)\,ds +G(X^{t,x;\theta}_T) \right],\ \ t\in[0,T].
 \label{eq-cost-funct}
\end{align}
Here and throughout this work, we use $E_{\sF_t}[\,\cdot\,]$ to denote the conditional expectation given $\sigma$-algebra $\sF_t$ for each $t\geq 0$. Then it is proved that the value function
\begin{align}
V(t,x)=\essinf_{\theta\in\cU}J(t,x;\theta),\quad t\in[0,T],
\label{eq-value-func}
\end{align}
 is a viscosity solution of the stochastic HJ equation \eqref{SHJB} (see \cite[Theorem 4.2]{qiu2017viscosity}).
 
 The stochastic HJ equation \eqref{SHJB}, because of the vanishing diffusion coefficients in the controlled differential equation \eqref{state-proces-contrl}, may be regarded as a degenerate case of fully nonlinear stochastic Hamilton-Jacobi-Bellman (HJB) equations that were first introduced by Peng \cite{Peng_92}. Peng proved the existence and uniqueness of weak solutions in Sobolev spaces for the superparabolic semilinear stochastic HJB equations in \cite{Peng_92}, while the wellposedness of general cases was claimed as an open problem, referring to Peng's plenary lecture of ICM 2010 \cite{peng2011backward}. In fact, the stochastic HJ equations are a class of backward stochastic partial differential equations (BSPDEs) which have been studied since  about forty years ago (see\cite{Pardoux1979}). The linear and semilinear BSPDEs have been extensively studied; we refer to \cite{DuQiuTang10,Hu_Ma_Yong02,ma2012non,Tang-Wei-2013} among many others. For the weak solutions and associated local behavior analysis for general quasi-linear BSPDEs, see \cite{QiuTangMPBSPDE11}, and we refer to \cite{Horst-Qiu-Zhang-14} for BSPDEs with singular terminal conditions. In the recent work \cite{Qiu2014weak}, the first author studied the weak solution in Sobolev spaces for a special class of the fully nonlinear stochastic HJB equations (with $\beta\equiv 0$ and $\sigma(t,x,v)\equiv v$).

More recently, a notion of viscosity solution was proposed in \cite{qiu2017viscosity} for general fully nonlinear stochastic HJB equations. In \cite{qiu2017viscosity}, the value function $V$ was verified to be the maximal viscosity solution under certain assumptions on the regularity of coefficients (see $(\cA^*)$ in Remark \ref{rmk-unique}), and further for the superparabolic cases when the diffusion coefficients $\sigma$ do not depend explicitly on $\omega\in\Omega$, the uniqueness is proved.  In this paper, we shall drop the strong assumptions on regularity of coefficients (see Remark \ref{rmk-unique}) and prove the uniqueness of viscosity solution to stochastic HJ equation \eqref{SHJB} corresponding to a degenerate fully nonlinear case of \cite{qiu2017viscosity}.

 Recalling heuristically the notion of viscosity solution proposed in \cite{qiu2017viscosity}, we may think of the concerned random fields like the first unknown variable $u$ and the value function $V$ as stochastic differential equations (SDEs) of the following form:
\begin{align}
u(t,x)=u(T,x)-\int_{t}^T\mathfrak{d}_{s} u(s,x)\,ds-\int_t^T\mathfrak{d}_{w}u(s,x)\,dW_s,\quad (t,x)\in[0,T]\times\bR^d. \label{SDE-u}
\end{align}
The Doob-Meyer decomposition theorem implies the uniqueness of the pair $(\mathfrak{d}_tu,\,\mathfrak{d}_{\omega}u)$ and thus makes sense of the linear operators $\mathfrak{d}_t$ and $\mathfrak{d}_{\omega}$ which actually coincide with the two differential operators introduced by Le$\tilde{\text{a}}$o, Ohashi and Simas in \cite[Theorem 4.3]{Leao-etal-2018}. In fact, an earlier discussion on operator $ \mathfrak{d}_{\omega}u$ may be found in \cite[Section 5.2]{cont2013-founctional-aop}. Through comparison, we have $\psi=\mathfrak{d}_{\omega}u$ and solving \eqref{SHJB} with a pair $(u,\psi)$ is equivalent to seeking $u$ (of form \eqref{SDE-u}) satisfying 
\begin{equation}\label{SHJB-eqiv}
  \left\{\begin{array}{l}
  \begin{split}
  -\mathfrak{d}_tu(t,x)
- \mathbb{H}(t,x,Du(t,x) )&=0,  \quad
                     (t,x)\in Q;\\
    u(T,x)&= G(x), \quad x\in\bR^d.
    \end{split}
  \end{array}\right.
\end{equation}
The equivalence relation between \eqref{SHJB} and \eqref{SHJB-eqiv} provides the key to defining the viscosity solutions for stochastic HJ equations. As a standard assumption in the general stochastic control theory, all the involved coefficients herein are only measurable w.r.t. $\omega$ on the sample space $(\Omega,\sF)$ and this challenge prevents us from defining the viscosity solutions in a point-wise manner, while motivating us to use a class of random fields of form \eqref{SDE-u} with sufficient spacial regularity as test functions. At each point $(\tau,\xi)$ ($\tau$ may be stopping time and $\xi$ may be an $\bR^d$-valued $\sF_{\tau}$-measurable variable) the classes of test functions are also parameterized by $\Omega_{\tau}\in\sF_{\tau}$. Another challenge is from the  nonanticipativity constraints on the unknown variables, which makes the classical variable-doubling techniques for deterministic HJ equations inapplicable in the proof of uniqueness for stochastic equations like \eqref{SHJB}. In this work, we first prove that the value function is the maximal viscosity (sub)solution which in fact reveals a weak version of comparison principle, and then through approximations, the value function is verified to be the unique one on basis of the established comparison results.

We refer to \cite{crandall2000lp,crandall1992user,juutinen2001definition,wang1992-I} among many others for the theory of (deterministic) viscosity solutions and \cite{buckdahn2007pathwise,LionsSouganidis1998b} for the stochastic viscosity solutions of (forward) SPDEs. Note that the (backward) stochastic HJB equations like  \eqref{SHJB} and the (forward) ones studied in  \cite{buckdahn2007pathwise,LionsSouganidis1998b} 
are essentially different, i.e., the noise term in the latter is exogenous, while in the former it is governed by the coefficients through the martingale representation and thus endogenous. 

When the coefficients $\beta,f$ and $G$ are deterministic functions of time $t$, control $\theta$ and the paths of $X$ and $W$, the optimal stochastic control problem is beyond the classical Markovian framework and the value function can be characterized by a path-dependent PDE. We refer to \cite{ekren2014viscosity,ekren2016viscosity-1,lukoyanov2007viscosity,peng2011note} for the theory of viscosity solutions of  such nonlinear path-dependent PDEs. In particular, in  \cite{ekren2014viscosity,ekren2016viscosity-1}, the authors applied the path-dependent viscosity solution theory to some classes of stochastic HJB equations which, however, required all the coefficients to be continuous in $\omega\in  \Omega$ due to the involved pathwise analysis. We would stress that, in the present work, all the involved coefficients are only measurable w.r.t. $\omega\in \Omega$ and we even do not need to specify any topology on $\Omega$, which allows the general random variables to appear in the coefficients.

The rest of this paper is organized as follows. In Section 2, we introduce in the first subsection some notations and the standing assumptions on the coefficients, and in the second subsection, the main result is exhibited. Two auxiliary results are presented in Section 3. Finally, Section 4 is devoted to the proof of our main result; we verify in the first subsection that the value function is the maximal viscosity solution  and then the uniqueness of viscosity solution is derived in the second subsection.

\section{Preliminaries and main result}

\subsection{Preliminaries}

Throughout this paper, we write $(s,y)\rightarrow (t^+,x)$, meaning that $s \downarrow t$ and $y\rightarrow x$.

Let $\mathbb B $ be a Banach space equipped with norm $\|\cdot\|_{\mathbb B }$. For each $t\in[0,T]$, denote by $L^0(\Omega,\sF_t;\mathbb B)$ the space of $\mathbb B$-valued $\sF_t$-measurable random variables. For $p\in[1,\infty]$, $\cS ^p ({\mathbb B })$ is the set of all the ${\mathbb B }$-valued,
 $\sP$-measurable continuous processes $\{\mathcal X_{t}\}_{t\in [0,T]}$ such
 that
{\small $$\|\mathcal X\|_{\cS ^p({\mathbb B })}:= \left\|\sup_{t\in [0,T]} \|\mathcal X_t\|_{\mathbb B }\right\|_{L^p(\Omega,\sF,\bP)}< \infty.$$
}
 Denote by $\mathcal{L}^p({\mathbb B })$ the totality of all  the ${\mathbb B }$-valued,
  $\sP$-measurable processes $\{\mathcal X_{t}\}_{t\in [0,T]}$ such
 that
 {\small
 $$
 \|\mathcal X\|_{\mathcal{L}^p({\mathbb B })}:=\left\| \bigg(\int_0^T \|\mathcal X_t\|_{\mathbb B }^p\,dt\bigg)^{1/p} \right\|_{L^p(\Omega,\sF,\bP)}< \infty.
 $$
 }
Obviously, $(\cS^p({\mathbb B }),\,\|\cdot\|_{\cS^p({\mathbb B })})$ and $(\mathcal{L}^p({\mathbb B }),\|\cdot\|_{\mathcal{L}^p({\mathbb B })})$
are Banach spaces.
For each $(k,q)\in \mathbb{N}_0\times [1,\infty]$ we define  the $k$-th Sobolev space $(H^{k,q},\|\cdot\|_{k,q})$ as usual, and for each domain $\cO\subset \bR^d$, denote by $C^k(\cO;\mathbb B)$ the space of $\mathbb B$-valued functions with the up to $k$-th order derivatives being bounded and continuous on $\cO$, $C^k_0(\cO;\mathbb B) $ being the subspace of $C^k(\cO;\mathbb B)$ vanishing on the boundary $\partial \cO$. If there is no confusion about $\mathbb B$, we will omit $\mathbb B$ and just write $C^k(\cO)$ and $C^k_0(\cO)$.   When $k=0$, write $C_0(\cO) $ and $C(\cO)$ simply. Through this paper, we define  $ C^{\infty}(\cO;\mathbb{B})=\cap_{k\in\bN^+} C^k(\cO;\mathbb{B})$ and
$$
C_0^{\infty}(\cO;\mathbb{B})=\cap_{k\in\bN^+} C_0^k(\cO;\mathbb{B}),\quad
\cS^p(C_{loc}(\bR^d))=\cap_{N>0} \cS^{p}(C(B_N(0))),\quad \text{ for } p\in[1,\infty].
$$


Throughout this work, we use the following assumption.

\bigskip
   $({\mathcal A} 1)$ \it $G\in L^{\infty}(\Omega,\sF_T;H^{1,\infty})$. For the coefficients $g=f,\beta^i$ $(1\leq i \leq d)$, \\
(i) $g:~\Omega\times[0,T]\times\bR^d\times U\rightarrow\bR$  {is}
$\sP\otimes\cB(\bR^d)\otimes\cB(U)\text{-measurable}$;\\
(ii) for almost all $(\omega,t)\in\Omega\times [0,T]$, $g(t,x,v)$ is uniformly continuous on $\bR^d\times U$;\\
(iii) there exists $L>0$ such that 
\begin{align*}
\|G\|_{L^{\infty}(\Omega,\sF_T;H^{1,\infty})} + 
\sup_{v\in\cU} \|g(\cdot,\cdot,v)\|_{\cS^{\infty}(H^{1,\infty})} \leq L.
\end{align*}
\rm

\subsection{Main Result}

We first introduce the test function space for viscosity solutions.

\begin{defn}\label{defn-testfunc}
For $u\in \cS^{2} (C_{loc}(\bR^d))$ with $Du\in \cL^2(C(\bR^d))$, we say $u\in \mathscr C_{\sF}^1$ if 
there exists $(\mathfrak{d}_tu, \,\mathfrak{d}_{\omega}u)\in \cL^2(C(\bR^d))\times  \cL^2(C(\bR^d))$ such that with probability 1
\begin{align*}
u(r,x)=u(T,x)-\int_r^T \mathfrak{d}_su(s,x)\,ds -\int_r^T\mathfrak{d}_{\omega}u(s,x)\,dW_s,\quad \forall\,(r,x)\in[0,T]\times\bR^d.
\end{align*}
\end{defn}

\begin{rmk}\label{rmk-test}
Instead of $ \mathscr C_{\sF}^2$ defined in \cite{qiu2017viscosity} which requires $D^2u $ and $ D \mathfrak{d}_{\omega}u$ to be lying in  $\cL^2(C(\bR^d))$, we use $ \mathscr C_{\sF}^1$ which imposes no requirement on $D^2u $ or $ D \mathfrak{d}_{\omega}u$. This is basically because the two terms $D^2u $ and $ D \mathfrak{d}_{\omega}u$ are not involved in the first-order BSPDE \eqref{SHJB}. Analogous to the space $ \mathscr C_{\sF}^2$ in \cite{qiu2017viscosity}, by Definition \ref{defn-testfunc}, we have in fact characterized the two linear operators $\mathfrak{d}_t$ and $\mathfrak{d}_{\omega}$ which is consistent with the two differential operators w.r.t. the paths of Wiener process $W$ in the sense of \cite{Leao-etal-2018}, defined via  a finite-dimensional approximation procedure based on controlled inter-arrival times and approximating martingales; in particular, for the operator $ \mathfrak{d}_{\omega}u$, an earlier discussion may be found in \cite[Section 5.2]{cont2013-founctional-aop}.   
We would also note that the operators $\mathfrak{d}_t$ and $\mathfrak{d}_{\omega}$ here are different from the path derivatives $(\partial_t,\,\partial_{\omega})$ via the functional It\^o formulas (see \cite{buckdahn2015pathwise} and \cite[Section 2.3]{ekren2016viscosity-1}). If $u(\omega,t,x)$ is smooth enough w.r.t. $(\omega,t)$ in the path space,  for each $x$,  we have the relation 
$$\mathfrak{d}_tu(\omega,t,x)=\left(\partial_t+\frac{1}{2}\partial^2_{\omega\omega}\right)u(\omega,t,x),\quad \mathfrak{d}_{\omega}u(\omega,t,x)=  \partial_{\omega}u(\omega,t,x),$$      which can be seen either from the applications in \cite[Section 6]{ekren2016viscosity-1} to BSPDEs or from a rough view on the pathwise viscosity solution of (forward) SPDEs in \cite{buckdahn2015pathwise}.
\end{rmk}


For each stopping time $t\leq T$, denote by $\mathcal{T}^t$ the set of stopping times $\tau$ satisfying $t\leq \tau \leq T$ and by $\mathcal{T}^t_+$ the subset of $\mathcal{T}^t$ such that $\tau>t$ for any $\tau\in \mathcal{T}^t_+$. For each $\tau\in\mathcal T^0$ and $\Omega_{\tau}\in\sF_{\tau}$, we denote by $L^0(\Omega_{\tau},\sF_{\tau};\bR^d)$ the set of $\bR^d$-valued $\sF_{\tau}$-measurable functions. 

We now introduce the notion of viscosity solutions. For each $(u,\tau)\in \cS^{2}(C_{loc}(\bR^d))\times \mathcal T^0$, $\Omega_{\tau}\in\sF_{\tau}$ with $\mathbb P(\Omega_{\tau})>0$ and $\xi\in L^0(\Omega_{\tau},\sF_{\tau};\bR^d)$, we define
\begin{align*}
\underline{\mathcal{G}}u(\tau,\xi;\Omega_{\tau})
:=\bigg\{
\phi\in\mathscr C^1_{\sF}:(\phi-u)(\tau,\xi)1_{\Omega_{\tau}}=0=\essinf_{\bar\tau\in\mathcal T^{\tau}} E_{\sF_{\tau}} \!\!\left[\inf_{y\in B_{\delta}(\xi)}
(\phi-u)(\bar\tau\wedge \hat{\tau},y)
\right]1_{\Omega_{\tau}}  \text{ a.s.}
&\\
\text{for some } (\delta,\hat\tau) \in (0,\infty)\times\mathcal T^{\tau}_+
\bigg\},&\\
\overline{\mathcal{G}}u(\tau,\xi;\Omega_{\tau}):=\bigg\{
\phi\in\mathscr C^1_{\sF}:(\phi-u)(\tau,\xi)1_{\Omega_{\tau}}=0=\esssup_{\bar\tau\in\mathcal T^{\tau}} \!\! E_{\sF_{\tau}}\left[\sup_{y\in B_{\delta}(\xi)}
(\phi-u)(\bar\tau\wedge \hat{\tau},y)
\right]1_{\Omega_{\tau}}  \text{ a.s.}
&\\
\text{for some } (\delta,\hat\tau) \in (0,\infty)\times\mathcal T^{\tau}_+
\bigg\}.&
\end{align*}
It is obvious that if $\underline{\mathcal{G}}u(\tau,\xi;\Omega_{\tau})$ or $\overline{\mathcal{G}}u(\tau,\xi;\Omega_{\tau})$ is nonempty, we must have $0\leq\tau <T$ on $\Omega_{\tau}$.  

Now it is at the stage to introduce the definition of viscosity solutions.

\begin{defn}\label{defn-viscosity}
We say $u\in \cS^2(C_{loc}(\bR^d))$ is a viscosity subsolution (resp. supersolution) of BSPDE \eqref{SHJB}, if $u(T,x)\leq (\text{ resp. }\geq) G(x)$ for all $x\in\bR^d$ a.s., and for any $\tau\in  \mathcal T^0$, $\Omega_{\tau}\in\sF_{\tau}$ with $\mathbb P(\Omega_{\tau})>0$ and $\xi\in L^0(\Omega_{\tau},\sF_{\tau};\bR^d)$ and any $\phi\in \underline{\cG}u(\tau,\xi;\Omega_{\tau})$ (resp. $\phi\in \overline{\cG}u(\tau,\xi;\Omega_{\tau})$), there holds
\begin{align}
&\text{ess}\liminf_{(s,x)\rightarrow (\tau^+,\xi)}
	E_{\sF_{\tau}} \left\{ -\mathfrak{d}_{s}\phi(s,x)-\bH(s,x,D\phi(s,x)) \right\}  \leq\ \,0, \text{ for almost all } \omega\in\Omega_{\tau}
\label{defn-vis-sub}
\end{align}
\begin{align}
\text{(resp.} \quad &\text{ess}\!\!\limsup_{(s,x)\rightarrow (\tau^+,\xi)} 
	\!\!E_{\sF_{\tau}} 
		\left\{ -\mathfrak{d}_{s}\phi(s,x)-\bH(s,x,D\phi(s,x)) \right\}  \geq\ \,0,  \text{ for almost all } \omega\in\Omega_{\tau}\text{).}
\label{defn-vis-sup}
\end{align}
 
The function $u$ is a viscosity solution of BSPDE \eqref{SHJB} if it is both a viscosity subsolution and a viscosity supersolution of \eqref{SHJB}.
\end{defn}

The stochastic HJ equation \eqref{SHJB} is a particular case of \cite[Theorem 4.2]{qiu2017viscosity} with vanishing diffusion coefficients. Therefore, as a straightforward consequence, we have the following existence of viscosity solution of BSPDE \eqref{SHJB}. 
\begin{thm}\label{thm-existence} (see \cite[Theorem 4.2]{qiu2017viscosity}). Let $(\cA1)$ hold. 
 The value function $V$ defined by \eqref{eq-value-func} is a viscosity solution of the stochastic Hamilton-Jacobi equation \eqref{SHJB}  in $\cS^{2}(C(\bR^d))$. 
\end{thm}
We note that even though the test function space used in \cite{qiu2017viscosity} is $\mathscr C^2_{\sF}$ instead of $\mathscr C^1_{\sF}$,  the proof of Theorem  \ref{thm-existence} follows exactly the same as that of \cite[Theorem 4.2]{qiu2017viscosity} as there would be no term involving $D^2u$ or $D \mathfrak{d}_{\omega}u$ in the proof.

Our main result is focused on the uniqueness.
\begin{thm}\label{thm-main} Let $(\cA1)$ hold. 
 The viscosity solution to stochastic HJ equation \eqref{SHJB} is unique in $\cS^2(C(\bR^d))$.
\end{thm}
\begin{rmk}\label{rmk-unique}
The uniqueness is twofold, consisting of the maximality and minimality of the  value function $V$ defined by \eqref{eq-value-func}. In  \cite[Theorem 5.2]{qiu2017viscosity}, it was concerned with the controlled \textit{stochastic} differential equation:
\begin{equation}\label{state-proces-contrl-SDE}
\left\{
\begin{split}
&dX_t=\beta(t,X_t,\theta_t)dt+\sigma(t,X_t,\theta_t)\,dW_t ,\,\,
\,t\in[0,T]; \\
& X_0=x,
\end{split}
\right.
\end{equation}
instead of the controlled \textit{ordinary} differential equation \eqref{state-proces-contrl} with random coefficients, and the value function was just proved to be the maximal viscosity (sub)solution which, however, relies on the following additional strong assumption on the coefficients:\\[4pt]
 $({\mathcal A}^*)$ There exists $q> 2+\frac{d}{2}$ such that $(G(\cdot),f(\cdot,\cdot,\theta) )  \in     L^{2}(\Omega,\sF_T;H^{q,2}) \times \cL^{2}(H^{q,2}) $  for  any $\theta\in\cU$, and  $g(\cdot,\cdot,\theta) \in\cL^{\infty}(H^{q,\infty}) $ for $g=\beta^i,\sigma^{ij}$ $(1\leq i \leq d,\,1\leq j\leq m)$.
\\[4pt] 
In fact, the author in \cite{qiu2017viscosity} only gave a complete uniqueness for superparabolic stochastic HJB equations with the diffusion coefficients depending only on time, state and control (see \cite[Theorem 5.6]{qiu2017viscosity}), while stochastic HJ equation \eqref{SHJB} has vanishing diffusion coefficients ($\sigma\equiv 0$) and thus is degenerate.
\end{rmk}

\section{Auxiliary Results}


In view of assumption $(\cA1)$ and the vanishing diffusion coefficients of stochastic differential equation \eqref{state-proces-contrl}, we may conclude the following assertions straightforwardly from \cite[Lemma 3.1]{qiu2017viscosity}.
\begin{lem}\label{lem-SDE}
Let $(\cA1)$ hold. Given $\theta\in\cU$, for the strong solution of SDE \eqref{state-proces-contrl}, there exists $K>0$  such that, for any $0\leq r \leq t\leq s \leq T$ and $\xi\in L^0(\Omega,\sF_r;\bR^d)$
 \\[3pt]
(i)   the two processes $\left(X_s^{r,\xi;\theta}\right)_{t\leq s \leq T}$ and $\left(X^{t,X_t^{r,\xi;\theta};\theta}_s\right)_{t\leq s\leq T}$ are indistinguishable;\\[2pt]
(ii)  $\max_{r\leq l \leq T} \left|X^{r,\xi;\theta}_l\right| \leq K \left(1+ |\xi|\right)$ a.s.;\\[2pt]
(iii) $\left| X^{r,\xi;\theta}_s-X^{r,\xi;\theta}_t  \right| \leq K \left(1+  |\xi|\right) (s-t)$ a.s.;\\[2pt]
(iv) given another $\hat{\xi}\in L^0(\Omega,\sF_r;\bR^d)$, 
$$
\max_{r\leq l \leq T} \left|X^{r,\xi;\theta}_l-X^{r,\hat\xi;\theta}_l\right| \leq K  |\xi-\hat\xi|
\quad \text{a.s.};
$$
(v) the constant $K$ depends only on $L$ and $T$.
\end{lem}

%

The following regular properties of the value function $V$ are from \cite[Proposition 3.3]{qiu2017viscosity}.
\begin{prop}\label{prop-value-func}
Let $(\cA 1)$ hold.
\\[4pt]
(i) For each  $t\in[0,T]$ and $\xi\in L^0(\Omega,\sF_t;\bR^d)$, there exists $\bar{\theta}\in\cU$ such that
$$
E\left[ J(t,\xi;\bar{\theta})-V(t,\xi)\right]<\eps.
$$
(ii) For each $(\bar{\theta},x)\in\cU\times \bR^d$, $\left\{J(t,X_t^{0,x;\bar\theta};\bar{\theta})-V(t,X_t^{0,x;\bar\theta})\right\}_{t\in[0,T]}$ is a supermartingale, i.e.,  for any $0\leq t\leq \tilde{t}\leq T$,
\begin{align}
V(t,X_t^{0,x;\bar\theta})
\leq E_{\sF_t}V(\tilde{t},X_{\tilde{t}}^{0,x;\bar{\theta}}) + E_{\sF_t}\int_t^{\tilde{t}}f(s,X_s^{0,x;\bar{\theta}},\bar\theta_s)\,ds,\,\,\,\text{a.s.}\label{eq-vfunc-supM}
\end{align}
(iii) For each $(\bar{\theta},x)\in\cU\times \bR^d$, $\left\{V(s,X_s^{0,x;\bar{\theta}})\right\}_{s\in[0,T]}$ is a continuous process.
\\[3pt]
(iv) 
 There exists $L_V>0$ such that for any $\theta\in\cU$
$$
|V(t,x)-V(t,y)|+|J(t,x;\theta)-J(t,y;\theta)|\leq L_V|x-y|,\,\,\,\text{a.s.},\quad \forall\,x,y\in\bR^d,
$$
with $L_V$ depending only on $T$ and the uniform Lipschitz constants of the coefficients $\beta,\sigma,f$ and $G$ w.r.t. the spatial variable $x$.
\\[3pt]
(v) With probability 1, $V(t,x)$ and $J(t,x;\theta)$ for each $\theta\in\cU$ are continuous  on $[0,T]\times\bR^d$ and 
$$ \sup_{(t,x)\in[0,T]\times\bR^d}   \max\left\{|V(t,x)|,\,|J(t,x;\theta)| \right\}       \leq L(T+1) \quad\text{a.s.}$$
\end{prop}


%



\section{Proof of Theorem \ref{thm-main}}

The proof   consists of two steps. In the first subsection, we prove that the value function is the maximal viscosity (sub)solution of the stochastic HJ  equation \eqref{SHJB}, which essentially yields a weak version of comparison principle. In the second subsection, the uniqueness is addressed on basis of the established comparison results through approximations.

Throughout this section, we define for any $\phi\in\mathscr C^1_{\sF}$ and $v\in U$,
\begin{align*}
\mathscr L^{v}\phi(t,x)=
 \mathfrak{d}_t \phi (t,x)      +\beta'(t,x,v)D\phi(t,x).
\end{align*}

\subsection{Maximal viscosity subsolution}
We first prove that the value function is the maximal viscosity (sub)solution of BSPDE \eqref{SHJB}. Such maximality is parallel to that of \cite[Theorem 5.2]{qiu2017viscosity}, but, as we want to achieve this without the additional strong regularity assumption (see $(\cA^*)$ in Remark \ref{rmk-unique}) required in \cite{qiu2017viscosity}, some new techniques are needed. The first one is based on smooth approximations.

Let
 \begin{equation}\label{bump-func}
 \rho(x)=
 \begin{cases}
 \tilde{c} \,  e^{\frac{1}{|x|^2-1}}&\quad \text{if } |x| < 1;\\
 0&\quad \text{otherwise};
 \end{cases}
 \quad \mbox{with}\quad \tilde{c}:=\left(     \int_{|x| < 1} e^{\frac{1}{x^2-1}}\,dx    \right)^{-1},
 \end{equation}
 and we define mollifier $\rho_l(x)=l^d\rho(lx)$, $x\in\bR^d$ for each $l\in\bN^+$. For $g=\beta^i(t,\cdot,v),\, f(t,\cdot,v),\, G(\cdot)$ $(1\leq i \leq d,\,1\leq j\leq m)$, take convolutions
 \begin{align*}
 g_l(x)=\int_{\bR^d} \rho_l(x-y)g(y)\,dy,\quad x\in\bR^d.
 \end{align*}
Then the coefficients $\beta_l$,  $f_l$ and $G_l$ satisfy assumption $(\cA 1)$ for each $l\in\bN^+$ and 
\begin{align}
\lim_{l\rightarrow \infty}\|G-G_l\|_{L^{\infty}(\Omega,\sF_T;L^{\infty}(\bR^d))} +\sup_{v\in\cU} \left( \|f_l-f\|_{\cS^{\infty}(L^{\infty}(\bR^d))}+\|\beta_l-\beta\|_{\cS^{\infty}(L^{\infty}(\bR^d))}\right)=0. \label{appr-l}
\end{align}
Moreover, since $G_l\in  L^{\infty}(\Omega,\sF_T;C^{k}(\bR^d))$, $\beta_l^i(\cdot,\cdot,\theta),f_l(\cdot,\cdot,\theta)\in \cS^{\infty} (C^{k}(\bR^d))$, $i=1,\dots,d$ for any $(k,\theta)\in\bN^+\times\cU$, by the classical solution theory for BSPDEs in \cite[Lemma 5.1, Theorems 4.6, 5.1\&5.2]{Tang_05} we have

\begin{prop}\label{prop-tang}
For each $(l,\theta)\in\bN^+\times \cU$, there exists a unique solution $u_l$ in $\mathscr C^1_{\sF} \cap \cS^{\infty}(C^2(\bR^d))$ to the following BSPDE:
 \begin{equation}\label{SHJB-dBSPDE-l}
  \left\{\begin{array}{l}
  \begin{split}
  -du_l(t,x)=\,&\displaystyle 
    \left[   \beta_l'(t,x,\theta_t)Du_l(t,x)
          +f_l(t,x,\theta_t)
               \right] \,dt -\psi_l(t,x)\, dW_{t}, \quad
                     (t,x)\in Q;\\
    u_l(T,x)=\, &G_l(x), \quad x\in\bR^d,
    \end{split}
  \end{array}\right.
\end{equation}
with $\psi_l=  \mathfrak{d}_{\omega} u_l$, and for each $x\in\bR^d$ and $0\leq t\leq s \leq T$, the random processes
$$
X_s^{t,x;\theta,l},\quad Y^{t,x;\theta,l}_s:=u_l(s,X_s^{t,x;\theta,l}) \quad \text{and} \quad Z^{t,x;\theta,l}_s:=\psi_l(s,X^{t,x;\theta,l}_s)
$$
satisfy the following forward-backward SDEs:
\begin{equation}\label{FBSDE-l}
\begin{cases}
X_s^{t,x;\theta}=x+\int_t^s \beta_l(r,X_r^{t,x;\theta},\theta_r)\,dr,\quad 0\leq t\leq s\leq T;\\
Y_s^{t,x;\theta}=G_l(X_T^{t,x;\theta}) +\int_s^T f_l(r,X_r^{t,x;\theta},\theta_r)\,dr-\int_s^T Z_r^{t,x;\theta}\,dW_r,\   0\leq t\leq s\leq T.
\end{cases}
\end{equation}
\end{prop}
\bigskip

We now introduce another two space-invariant stochastic processes. Put
\begin{align*}
&
	\delta G_l:=\esssup_{x\in\bR^d} \left|G_l(x)-G(x) \right|,
\\
& 
	\delta f_t^{l,\theta} :=\sup_{x\in\bR^d}
		\left|f_l(t,x,\theta_t)-f(t,x,\theta_t)\right|,\quad \text{for } t\in  [0,T],
\\
&
 	\delta \beta^{l,\theta}_t :=\sup_{x\in\bR^d}
		\left|\beta_l( t,x,\theta_t)-\beta(t,x,\theta_t)\right|,\quad \text{for } t\in  [0,T].
\end{align*}
Let $(Y^l,Z^l)\in \cS^2(\bR)\times\cL^2(\bR^m)$ be the solution to BSDE:
\begin{align*}
Y^l_t=\delta G_l+\int_t^T (\delta f_s^{l,\theta}+ K\delta \beta_s^{l,\theta})\,ds-\int_t^T Z_s^{l}\,dW_s
\end{align*}
where $K=L_V$ is the constant from (iv) of Proposition \ref{prop-value-func}. Recalling relation \eqref{appr-l}, we have by the theory of BSDE that
\begin{align}
\|Y^l\|_{\cS^{2}(\bR)} \leq C \delta_l, \quad\text{with }\delta_l=\|\delta G_l\|_{L^{2}(\Omega,\sF_T)} +  \|\delta f^{l,\theta}+K\delta \beta^{l,\theta}\|_{\cL^{2}(\bR)}\rightarrow 0,
\label{appr-Y}
\end{align}
where the constant $C$ is independent of $l$.

\begin{lem}\label{lem-approx-j}
For each $\theta\in\cU$, setting $\hat{J}_l(t,x)=u_l(t,x)+Y_t^l$, we have $\hat{J}_l(t,x)\in \mathscr C^1_{\sF}$  with
$$
\lim_{l\rightarrow \infty} \| \hat{J}_l(t,x)-J(t,x;\theta) \|_{\cS^{2}(C(\bR^d))} =0
$$
and 
\begin{align}
\text{ess}\liminf_{(s,x)\rightarrow (t^+,y)}  E_{\sF_t}\left[ -\mathfrak{d}_{s}\hat{J}_l(s,x)-\bH(s,x,D\hat J_l(s,x) ) \right] 
		\geq 0, \ \text{a.s.}
		\quad \forall (t,y)\in[0,T)\times\bR^d. \label{relation-sup}
\end{align}
\end{lem}
\begin{proof}
It follows obviously that  $\hat{J}_l(t,x)\in \mathscr C^1_{\sF}$ from Proposition \ref{prop-tang}.

Using Gronwall's inequality through standard computations gives for any $(s,x)\in[0,T]\times \bR^d$,
\begin{align*}
 \sup_{s\leq t\leq T} \left|  X^{s,x;\theta,l}_t-X^{s,x;\theta}_t   \right|
 &\leq C \int_s^T  \left| \beta_l\left( X^{s,x;\theta,l}_t,\theta_t\right)-\beta\left(t,X^{s,x;\theta,l}_t,\theta_t\right) \right|\,dt
	\\
&\leq
C \delta_l,\quad \text{a.s.}
\end{align*}
with constant $C$ depending only on $L$ and $T$. Then
\begin{align*}
&\left| u_l(s,x)-J(s,x;\theta)\right|
\\
&	
	\leq
		E_{\sF_s}\bigg[ 
		\int_s^T\Big( \delta f^{l,\theta}_t +\Big| f\left(t,X^{s,x;\theta,l}_t,\theta_t\right) -f\left(t,X^{s,x;\theta}_t,\theta_t\right)\Big| \Big)\,dt
\\
&\quad\quad\quad
		+\delta G_l +\Big| G\left(X^{s,x;\theta,l}_T\right) -G\left(X^{s,x;\theta}_T \right)\Big|
		\bigg] 
\\
&
	\leq |Y^{l}_s|
		 +2L(T+1) E_{\sF_s}\left[\sup_{s\leq t\leq T}   \left|  X^{s,x;\theta,l}_t-X^{s,x;\theta}_t   \right|    \right]
\\
&
	\leq C \delta_l, \quad\text{a.s.}
\end{align*}
with the constant $C$ independent of $l,s$ and $x$, which together with \eqref{appr-Y} implies
$$\lim_{l\rightarrow \infty} \| \hat{J}_l(t,x)-J(t,x;\theta) \|_{\cS^{2}(C(\bR^d))} =0.$$

Notice that  the coefficients $\beta_l$,  $f_l$ and $G_l$ satisfy assumption $(\cA 1)$ with the identical constant $L$. In view of Proposition \eqref{appr-Y} and the BSDE for $Y^l$, the random field $u_l$ satisfies (iv) of Proposition \ref{prop-value-func} with the same Lipschitz constant $L_V$, and we have
\begin{align*}
-\mathfrak{d}_{t}  u_l  
	&=
	     \beta_l'Du_l
	  +f_l\\
	-\mathfrak{d}_{t}  Y^l 
	  &=\delta f^{l,\theta}+ L_V  \delta\beta^{l,\theta}
\end{align*}
and thus
\begin{align*}
-\mathscr L^{\theta_s} \hat{J}_l-f
&=-\mathfrak{d}_{t}  \hat{J}_l
	-
	     \beta_l'D\hat{J}_l
	  -f_l
	  - \left(\beta-\beta_l\right)'  D\hat{J}_l      
	  -f+f_l
	  \\
&=-\mathfrak{d}_{t}  u_l +\delta f^{l,\theta}+ L_V  \delta\beta^{l,\theta} 
	-
	     \beta_l'Du_l
	  -f_l
	  - \left(\beta-\beta_l\right)'  Du_l      
	  -f+f_l
	  \\
&= \delta f^{l,\theta}+ L_V  \delta\beta^{l,\theta} 
	  - \left(\beta-\beta_l\right)'  Du_l      
	  -f+f_l	  
	  \\
	  &\geq 0,	  
\end{align*}
where we omitted the inputs for each involved function for the sake of convenience.
Therefore, it holds that
\begin{align*}
\text{ess}\liminf_{(s,x)\rightarrow (t^+,y)}  E_{\sF_t}\left[ -\mathfrak{d}_{s}\hat{J}_l(s,x)-\bH(s,x,D\hat J_l(s,x) ) \right] 
		\geq 0, \ \text{a.s.}
		\quad \forall (t,y)\in[0,T)\times\bR^d.
\end{align*}

\end{proof}


Recalling the compactly-supported smooth (bump) function $\rho(x)$ defined in \eqref{bump-func}, set
\begin{align*}
h(x)=\int_{\bR^d} 1_{\{|y|>1 \}} \left(  |y|-1\right) \rho(x-y)\,dy,\quad x\in\bR^d. 
\end{align*}
%
Then the function $h(x)$ is convex and continuously differentiable with $h({0})=0$, $h(x)>0$ whenever $|x|>0$, and 
\begin{align}
h(x)> |x|-2,\quad    |Dh(x)|\leq 1 \text{ for any } x\in\bR^d.\label{h-linear-growth}
\end{align}

\begin{thm}\label{thm-max}
Let $(\cA1)$ hold.   Let $u\in \cS^{2}(C(\bR^d))$ be a viscosity subsolution  of the stochastic HJB equation \eqref{SHJB}. It holds a.s. that $u(t,x)\leq V(t,x)$ for any $(t,x)\in[0,T]\times\bR^d$, where $V$ is the value function defined by \eqref{eq-value-func}.
\end{thm}

\begin{proof}
We argue by contradiction.  Suppose that with a positive probability, $u(t,\bar x)>V(t,\bar x)$  at some point $(t,\bar x)\in [0,T)\times \bR^d$.  In view of the approximating relations between $V(t,x)$, $J(t,x;\theta)$ and $\hat{J}_l(t,x)$ in Proposition \ref{prop-value-func} and Lemma \ref{lem-approx-j}, we have some $(l,\theta)\in\bN^+\times \cU$ such that $u(t,\bar x)>\hat{J}_l(t,\bar x)$ with a positive probability; more precisely, there exists $\kappa>0$ such that $\bP( \overline\Omega_t)>0$ with $\overline\Omega_t:=\{u(t,\bar x)-\hat{J}_l(t,\bar x) >\kappa  \}$. Furthermore,  for any $\eps \in (0,1)$, there exists $\xi_t\in L^0(\overline\Omega_t,\sF_t;\bR^d)$ such that
$$
\alpha:=u(t,\xi_t)-\hat J_l(t,\xi_t)-\eps h(\xi_t-\bar x )=\max_{x\in\bR^d} \{  u(t,x)-\hat J_l(t,x) -\eps h(x-\bar x )\}\geq \kappa\text{ for almost all }\omega\in\overline\Omega_t,
$$ 
where the existence and the measurablity  of $\xi_t$ fellow from the measurable selection, the linear growth of function $h(x)$ (see \eqref{h-linear-growth}) and the fact  that $u,\hat J_l\in \cS^{2}(C(\bR^d))$. Note that $\kappa$ and $\overline\Omega_t$ are independent of $\eps$. W.l.o.g, we take $\overline\Omega_t=\Omega$ in what follows.

For each $s\in(t,T]$, choose an $\sF_s$-measurable variable $\xi_s$ such that  
\begin{align}
\left( u(s,\xi_s)-\hat J_l(s,\xi_s) -\eps h( \xi_s-\bar x )\right)^+=\max_{x\in\bR^d}   \left(u(s,x)-\hat J_l(s,x)-\eps h(x-\bar x )\right)^+ . \label{eq-maxima}
\end{align}  
Set
\begin{align*}
Y_s
&=
	(u(s,\xi_s) -\hat J_l(s,\xi_s)-\eps h(\xi_s-\bar x ))^++\frac{\alpha (s-t)}{2(T-t)};\\
Z_s
&= 
	\esssup_{\tau\in\cT^s} E_{\sF_s} [Y_{\tau}],
\end{align*}
where we recall that  $\mathcal{T}^s$ denotes the set of stopping times valued in $[s,T]$. As $u,\hat J_l\in \cS^2(C(\bR^d))$, it follows obviously the time-continuity of  
$$
	\max_{x\in\bR^d}   \left(u(s,x)-\hat J_l(s,x)-\eps h(x-\bar x )\right)^+
$$ 
and thus that of
$\left( u(s,\xi_s)-\hat J_l(s,\xi_s) -\eps h(\xi_s-\bar x )\right)^+$. Therefore, the process $(Y_s)_{t\leq s \leq T}$ has continuous trajectories. 
Define $\tau=\inf\{s\geq t:\, Y_s=Z_s\}$. In view of the optimal stopping theory, observe that
$$
E_{\sF_t}Y_T=\frac{\alpha}{2}
	<\alpha=Y_t\leq Z_t=E_{\sF_t}Y_{\tau} =E_{\sF_t}Z_{\tau}.
$$
It follows that $\bP(\tau<T)>0$. As 
$$
	(u(\tau,\xi_{\tau}) -\hat J_l(\tau,\xi_{\tau})-\eps h(\xi_{\tau}-\bar x))^++\frac{\alpha (\tau-t)}{2(T-t)}
=Z_{\tau}\geq E_{\sF_{\tau}}[Y_T]= \frac{\alpha}{2},
$$
we have 
\begin{align}
\bP((u(\tau,\xi_{\tau}) -\hat J_l(\tau,\xi_{\tau})-\eps h(\xi_{\tau}-\bar x))^+>0)>0. \label{relation-tau}
\end{align}
 Define 
$$\hat\tau=\inf\{s\geq\tau:\, (u(s,\xi_{s}) -\hat J_l(s,\xi_{s})-\eps h(\xi_{s}-\bar x))^+\leq 0\}.$$ Obviously, $\tau\leq\hat\tau\leq T$.
 Put $\Omega_{\tau}=\{\tau<\hat\tau\}$. Then $\Omega_{\tau}\in \sF_{\tau}$ and in view of relation \eqref{relation-tau}, and the definition of $\hat\tau$, we have $\bP(\Omega_{\tau})>0$.

Set 
$$\phi(s,x)=\hat J_l(s,x)+\eps h(x-\bar x) -\frac{\alpha (s-t)}{2(T-t)}+E_{\sF_s}Y_{\tau}
.$$
 Then $\phi\in\mathscr{C}^1_{\sF}$ since $\hat J_l \in \mathscr{C}^1_{\sF}$. For each $\bar\tau\in\cT^{\tau}$, \footnote{Recall that $\cT^{\tau}$ denotes the set of stopping times $\zeta$ satifying $\tau\leq \zeta\leq T$ as defined in Section 2.2.} we have for almost all $\omega\in\Omega_{\tau}$, 
\begin{align*}
\left( \phi-u\right)(\tau,\xi_{\tau})
=0=Y_{\tau}-Z_{\tau}
\leq 
Y_{\tau}- E_{\sF_{\tau}}\left[Y_{\bar\tau\wedge \hat{\tau}} \right] 
= E_{\sF_{\tau}} \left[ \inf_{y\in\bR^d} (\phi-u)(\bar\tau\wedge\hat\tau,y)   \right],
\end{align*}
which together with the arbitrariness of $\bar\tau$ implies that $\phi\in \underline{\cG} u(\tau,\xi_{\tau};\Omega_{\tau})$. As $u$ is a viscosity subsolution, by Lemma \ref{lem-approx-j} it holds that for almost all $\omega\in\Omega_{\tau}$, 
\begin{align*}
0
&\geq
	 \text{ess}\liminf_{(s,x)\rightarrow (\tau^+,\xi_{\tau})}
	E_{\sF_{\tau}} \left\{ -\mathfrak{d}_{s}\phi(s,x)-\bH(s,x,D\phi(s,x) ) \right\}
\\
&=
	\frac{\alpha}{2(T-t)}
\\
&\quad\quad
	+ \text{ess}\liminf_{(s,x)\rightarrow (\tau^+,\xi_{\tau})}
	E_{\sF_{\tau}} \left\{  
	-\mathfrak{d}_{s}\hat J_l(s,x)-\bH(s,x,D \hat J_l(s,x) +\eps Dh(x-\bar x)) \right\}
\\
&\geq
	\frac{\kappa}{2(T-t)} + \text{ess}\liminf_{(s,x)\rightarrow (\tau^+,\xi_{\tau})}
	E_{\sF_{\tau}} \left\{  
	-\mathfrak{d}_{s}\hat J_l(s,x)-\bH(s,x,D \hat J_l(s,x)  ) \right\}  \\
	&
	\quad\quad\quad - \eps E_{\sF_{\tau}}\left[   \sup_{(s,x,v)\in [0,T]\times \bR^d \times v}   \left|\beta_l(s,x,v)\right| \cdot |Dh(x-\bar x)|  \right]	\\
&\geq \frac{\kappa}{2(T-t)}  -\eps \, L.
\end{align*}
This is an obvious contradiction as $\eps$ is sufficiently small. 
\end{proof}
\begin{rmk}
Compared with the proof of \cite[Theorem 5.2]{qiu2017viscosity}, we added two new techniques in the above proof: (i) due to the lack of regularity of coefficients, we construct sequences $\{\hat J_l\}$ in $\mathscr C_{\sF}^1$ to approximate $J(\cdot,\cdot;\theta)$; (ii) lack of spatial integrability of $V$ (or $V$ possibly being nonzero at infinity) motivates us to introduce a penalty function $h$ in the proof to ensure the existence of maximums (for instance, in \eqref{eq-maxima}).
\end{rmk}
 
Throughout the proof of Theorem \ref{thm-max}, we see that only the viscosity subsolution property of $u$ and the property \eqref{relation-sup} of $\hat J_l \in\mathscr C^1_{\sF}$ are used. Hence, omitting the proofs we have the following weak version of comparison principle.

\begin{cor}\label{cor-cmp}
Let $(\cA1)$ hold and $u$ be a viscosity subsolution (resp. supersolution)  of BSPDE \eqref{SHJB}  and $\phi\in\mathscr C^1_{\sF}$, $\phi(T,x)\geq (\text{resp. }\leq) G(x)$ for all $x\in\bR^d$ a.s. and with probability 1
{\small
\begin{align*}
\text{ess}\liminf_{(s,x)\rightarrow (t^+,y)}
	E_{\sF_{t}} \left\{  -\mathfrak{d}_{s}\phi(s,x)-\bH(s,x,D \phi(s,x) ) \right\} \geq 0
\\
\text{(resp. }
\text{ess}\limsup_{(s,x)\rightarrow (t^+,y)}
	E_{\sF_{t}} \left\{  -\mathfrak{d}_{s}\phi(s,x)-\bH(s,x,D \phi(s,x) ) \right\} \leq 0\text{)}
\end{align*}
}
for  all $(t,y)\in[0,T)\times\bR^d$. It holds a.s. that $u(t,x)\leq$ (resp., $\geq$) $\phi(t,x)$,  $\forall \, (t,x)\in[0,T]\times\bR^d$.
\end{cor}

\subsection{Uniqueness of viscosity solution}
 We shall prove the uniqueness on basis of the established comparison results. First, we approximate the coefficients $\beta,\,f$ and $G$ via regular functions.
\begin{lem}\label{lem-approx}
 Let $(\cA1)$ hold. For each $\eps>0$, there exist partition $0=t_0<t_1<\cdots<t_{N-1}<t_N=T$ for some $N>3$ and functions 
$$(G^N,f^N,\beta^N)\in C^{3}(\bR^{m_0\times N}\times\bR^d)\times C( U;C^3([0,T]\times\bR^{m_0\times N}\times\bR^d)) \times C(U;C^3([0,T]\times\bR^{m_0\times N}\times\bR^d))$$  
such that 
\begin{align*}
&
	G^{\eps}:=\esssup_{x\in\bR^d} \left|G^N( W_{t_1},\cdots, W_{t_N},x)-G(x) \right|,
\\
& 
	f^{\eps}_t :=\esssup_{(x,v)\in\bR^d\times U}
		\left|f^N( W_{t_1\wedge t},\cdots, W_{t_N\wedge t},t,x,v)-f(t,x,v)\right|,\quad \text{for } t\in  [0,T],
\\
&
 	\beta^{\eps}_t :=\esssup_{(x,v)\in\bR^d\times U}
		\left|\beta^N( W_{t_1\wedge t},\cdots, W_{t_N\wedge t},t,x,v)-\beta(t,x,v)\right|,\quad \text{for } t\in  [0,T],
\end{align*}
 are $\sF_t$-adapted with
\begin{align*}
	\left\| G^{\eps}  \right\|_{L^2(\Omega,\sF_T;\bR)} + \left\| f^{\eps}  \right\|_{\cL^2(\bR)}   + \left\| \beta^{\eps}  \right\|_{\cL^2(\bR^d)}    <\eps,
\end{align*}
and $G^N$, $f^N$ and $\beta^N$ are uniformly Lipschitz-continuous in the space variable $x$ with an identical Lipschitz-constant $L_c$ independent of $N$ and $\eps$.
\end{lem}
  Although the proof of Lemma \ref{lem-approx} is an application of standard density arguments, we would sketch the proof for the readers who are interested.
  
  \begin{proof}[Sketched proof of Lemma \ref{lem-approx}]
  Consider the approximations of $f$.  First,
  in a similar way to \cite[case (c) in the proof of Proposition 2.2, Page 29]{gawarecki2010stochastic}, the dominated convergence theorem indicates that $f$ may be approximated in $\cL^2(C(U\times \bR^d))$ by random fields of the form:
  $$ \bar f^l(\omega,t,x,v)= \phi_1(\omega, x,v)1_{[0,t_1]}(t)+\sum_{j=2}^l\phi_j(\omega,x,v) 1_{(t_{j-1},t_j]}(t), $$
  where $0=t_0< t_1< \cdots< t_l<T$, and for $j=1,\dots,l$, $\phi_j=f(t_{j-1},\cdot,\cdot,\cdot)\in L^2(\Omega,\sF_{t_{j-1}}; C (U\times\bR^d))$. In fact, with the identity approximations as in \eqref{appr-l}, we may take instead 
  $$\phi_j\in L^2(\Omega,\sF_{t_{j-1}};
   C (U, C^{\infty}(\bR^d)),\quad j=1,\dots, l.$$ Further, for each $j\geq 2$, $\phi_j$ may be approximated \textit{monotonically} (see \cite[Lemma 1.2, Page 16]{da2014stochastic} for instance) by simple random variables of the following form:
  $$
  \sum_{i=1}^{l_j}1_{A_i^j} (\omega) h^j_i(x,v),\quad \text{with }h_i^j\in C(U;C^{\infty}(\bR^d)),\quad A_i^j\in\sF_{t_{j-1}}, \quad i=1,\dots, l_j,
  $$
  and  \cite[Lemma 4.3.1., page 50]{oksendal2003stochastic} implies that each $1_{A_i^j}$ may be approximated in $L^2(\Omega,\sF_{t_{j-1}})$ by functions in the following set
  $$
  \{g(W_{\tilde t_1},\dots,W_{\tilde t_{l^j_i}} ):\, \tilde t_r\in[0,t_{j-1}],\,\,g\in  C^{\infty}_0(\bR^{l_i^j}) \}.
  $$ 
  In addition, each $1_{(t_{j-1},t_j]}$ may be increasingly approximated by compactly-supported nonnegative functions $\varphi_j\in C^{\infty}((t_{j-1},T];\bR)$. To sum up, $f$ may be approximated in $\cL^2(C(U\times \bR^d))$ by the following random fields:
  \begin{align*}
  f^N( W_{\bar t_1\wedge t},\cdots, W_{\bar t_N\wedge t},t,x,v)
  =\sum_{j=1}^l\sum_{i=1}^{l_j} g_i^j(W_{\bar t_1\wedge t_{j-1}},\cdots, W_{\bar t_N\wedge t_{j-1}})h_i^j(x,v) \varphi_j(t),
    \end{align*}
  where $0=\bar t_0 <\bar t_1<\cdot<\bar t_{N-1}<\bar t_{N}=T$, and $g_i^j,\,h_i^j,\, \varphi_j$ are smooth functions. The required approximations for $G$ and $\beta$ follow similarly. 
  \end{proof}

    We are now ready to present the proof for the uniqueness of viscosity solution.

\begin{proof}[Proof of Theorem \ref{thm-main}]

Define 
\begin{align*}
\overline{\mathscr V}=
\bigg\{
	\phi\in\mathscr C^1_{\sF}:&\, \,\phi(T,x)\geq G(x)\,\,\forall x\in\bR^d, \text{ a.s., and with probability 1,}\\
&
	\text{ess}\liminf_{(s,x)\rightarrow (t^+,y)}	E_{\sF_t}\left[ -\mathfrak{d}_{s}\phi(s,x)-\bH(s,x,D\phi(s,x) ) \right]
	\geq 0,	\quad \forall (t,y)\in[0,T)\times\bR^d
	\bigg\}\\
\underline{\mathscr V}=
\bigg\{
	\phi\in\mathscr C^1_{\sF}:&\, \phi(T,x)\leq G(x)\,\,\,\forall x\in\bR^d, \text{ a.s., and with probability 1,}\\
&
		\text{ess}\limsup_{(s,x)\rightarrow (t^+,y)}  E_{\sF_t}\left[ -\mathfrak{d}_{s}\phi(s,x)-\bH(s,x,D\phi(s,x) ) \right] 
		\leq 0,\quad \forall (t,y)\in[0,T)\times\bR^d
	\bigg\},
\end{align*}
and set
\begin{align*}
\overline{u}=\essinf_{\phi\in \overline{\mathscr V}} \phi, \quad 
\underline{u}=\esssup_{\phi\in \underline{\mathscr V}} \phi.
\end{align*}
In view of Corollary \ref{cor-cmp}, for any viscosity solution $u\in \cS^{2}(C(\bR^d))$ we have $\underline u\leq u\leq \overline u$.  Therefore,  for the uniqueness of viscosity solution, it is sufficient to check $\underline u=V= \overline u$. 

Let $(\Omega',\sF',\{\sF'_t\}_{t\geq 0}, \bP')$ be another complete filtered probability space which carries a d-dimensional standard Brownian motion $B=\{B_t\, :\, t\geq 0\}$ with $\{\sF'_t\}_{t\geq 0}$ generated by $B$ and augmented by all the $\bP'$-null sets in $\sF'$. Set
$$
(\bar\Omega,\bar\sF,\{\bar\sF_t\}_{t\geq 0},\bar\bP)=
(\Omega\times\Omega',\sF \otimes\sF',\{\sF_t\otimes\sF'_t \}_{t\geq 0},\bP \otimes\bP').
$$
Then $B$ and $W$ are independent on $(\bar\Omega,\bar\sF,\{\bar\sF_t\}_{t\geq 0},\bar\bP)$ and it is easy to see that all the theory established in previous sections still hold on the enlarged probability space.


For each fixed $\eps\in(0,1)$, choose $(G^{\eps},\,f^{\eps},\,\beta^{\eps})$ and $(G^N,f^N,\beta^N)$ as in Lemma \ref{lem-approx}. Recalling the standard theory of backward SDEs (see \cite{Hu_2002} for instance), let the  pairs  $(Y^{\eps},Z^{\eps})\in \cS^2_{\sF}(\bR)\times \cL^2_{\sF}(\bR^{m})$ and $(y_t,z_t)\in  \cS^2_{\sF'}(\bR)\times \cL^2_{\sF'}(\bR^{d})$   be the solutions  of backward SDEs
$$
Y_s^{\eps}=G^{\eps}+
	\int_s^T\left(f^{\eps}_t+K\beta^{\eps}_t\right)\,dt
		-\int_s^TZ^{\eps}_s\,d W_s,
$$
 and 
 $$
 y_s=
 |B_T|+\int_{s}^T |B_r| \,dr-\int_s^T z_r\,dB_r,
 $$
respectively, and for each $(s,x)\in[0,T)\times\bR^d$,  set
{\small
\begin{align*}
{V}^{\eps}(s,x)
&=\essinf_{\theta\in\cU} E_{\bar\sF_s} \left[
	\int_s^Tf^N\left( W_{t_1\wedge t},\cdots, W_{t_N\wedge t},t,X^{s,x;\theta,N}_t,\theta_t\right)\,dt
		+G^N\left( W_{t_1},\cdots, W_{t_N},X^{s,x;\theta,N}_T\right)
		\right],
\end{align*}
}
where the constant $K\geq 0$ is to be determined later and $X^{s,x;\theta,N}_t$ satisfies the SDE
\begin{equation*}
\left\{
\begin{split}
&dX_t=\beta^N(t,X_t,\theta_t)dt +\delta_N dB_t,\,\,
\,t\in[s,T]; \\
& X_s=x
\end{split}
\right.
\end{equation*}
with $\delta_N>0$ being a constant.
 
 By the viscosity solution theory of fully nonlinear parabolic PDEs (see \cite[Theorems I.1 and II.1]{lions-1983} for instance),
  when $s\in[t_{N-1},T)$, 
$$V^{\eps}(s,x)=\tilde V^{\eps}(s,x,  W_{t_1},\cdots, W_{t_{N-1}},{W}_s)$$ 
with
{\small
\begin{align*}
&\tilde{V}^{\eps}(s,x,  W_{t_1},\cdots, W_{t_{N-1}},y)\\
&=\essinf_{\theta\in\cU} E_{\bar\sF_s,  W_s=y} \left[
	\int_s^Tf^N\left(  W_{t_1},\cdots, W_{t_{N-1}}, W_{t_N\wedge t},t,X^{s,x;\theta,N}_t,\theta_t\right)\,dt
		+G^N\left( W_{t_1},\cdots, W_{t_N},X^{s,x;\theta,N}_T\right)
		\right]
\end{align*}
}
satisfying the HJB equation of the following form
{\small
\begin{equation}\label{HJB-N}
  \left\{\begin{array}{l}
  \begin{split}
  -D_tu(t,x,y)=\,& 
  \frac{1}{2} \text{tr}\left(D_{yy}u(t,x,y)\right) 
  +\frac{\delta_N^2}{2} \text{tr}\left( D_{xx}u(t,x,y)\right) \\
  &+\essinf_{v\in U} \bigg\{   
		  (\beta^N)'(  W_{t_1},\cdots, W_{t_{N-1}},y,t,x,v)D_xu(t,x,y)\\
	  &+f^N(W_{t_1},\cdots, W_{t_{N-1}},y,t,x,v)\bigg\},
                     \quad (t,x,y)\in [t_{N-1},T)\times\bR^d\times\bR^{m_0};\\
    u(T,x,y)=\, &G^N(  W_{t_1},\cdots, W_{t_{N-1}},y,x), \quad (x,y)\in\bR^d\times\bR^{m}.
    \end{split}
  \end{array}\right.
\end{equation}
}
%
%
and thus the regularity theory of viscosity solutions (see \cite[Theorem 1.1]{krylov1982boundedly} or \cite[Chapter 6]{Krylov-1987} for instance\footnote{As $U\subset\bR^n$ is a nonempty compact set, it has a denumerable subset $\mathcal K\subset U$ that is dense in $U$, and by the continuity of the coefficients, the essential infimum may be taken over $\mathcal K$. This together with some basic properties of viscosity solutions (see \cite[Proposition 3.7]{wang1992-I} for instance) allows  \cite[Theorem 1.1]{krylov1982boundedly} to be applied straightforwardly.}) gives 
$$\tilde{V}^{\eps}(\cdot,\cdot,  W_{t_1},\cdots, W_{t_{N-1}},\cdot)\in 
L^{\infty}\left(\Omega,\sF_{t_{N-1}}; C^{1+\frac{\bar\alpha}{2},2+\bar\alpha}([t_{N-1},T]\times\bR^d) \right) 
,
$$
for some $\bar\alpha \in (0,1)$, where the \textit{time-space} H\"older space $C^{1+\frac{\bar\alpha}{2},2+\bar\alpha}([t_{N-1},T]\times\bR^d)$ is defined as usual. We can make similar arguments on time interval $[t_{N-2},t_{N-1})$ taking the obtained $V^{\eps}(t_{N-1},x)$ as the terminal value, and recursively on intervals $[t_{N-3},t_{N-2})$, $\dots$, $[0,t_{1})$.
Furthermore, applying the  It\^o-Kunita formula to $\tilde{V}^{\eps}(s,x, \tilde W_{t_1},\cdots,\tilde W_{t_{N-1}},y)$ on $[t_{N-1},T]$ yields that
 {\small
 \begin{equation}\label{SHJB-N}
  \left\{\begin{array}{l}
  \begin{split}
  -dV^{\eps}(t,x-\delta_N B_t)=\,& 
  \essinf_{v\in U} \bigg\{
  (\beta^N)'( W_{t_1},\cdots, W_{t_{N-1}}, W_t,t,x-\delta_N B_t,v) D_xV^{\eps}(t,x-\delta_N B_t)\\
&	
  +f^N( W_{t_1},\cdots, W_{t_{N-1}}, W_t,t,x-\delta_N B_t,v)\bigg\}\,dt\\
&
		-D_{y} \tilde{V}^{\eps}(s,x, \tilde W_{t_1},\cdots,\tilde W_{t_{N-1}},W_t) \,d W_t +\delta_N D_{x} V^{\eps}(t,x-\delta_N B_t)\,dB_t,\\
		&
                     \quad (t,x)\in [t_{N-1},T)\times\bR^d;\\
    V^{\eps}(T,x)=\, &G^N(  W_{t_1},\cdots, W_{t_{N-1}}, W_T,x-\delta_N B_T), \quad x\in\bR^d.
    \end{split}
  \end{array}\right.
\end{equation}
}
It follows similarly on intervals $[t_{N-2},t_{N-1})$, $\dots$, $[0,t_{1})$, and finally we have $V^{\eps}(\cdot,\cdot-\delta_N B_{\cdot}) \in \mathscr C^1_{\bar\sF}$.

In view of the approximation in Lemma \ref{lem-approx} and with an analogy to the proof of (iv) in Proposition \ref{prop-value-func}, there exists $\tilde{L}>0$ such that 
$$
\max_{(t,x)\in[0,T]\times \bR^d}\left\{ |DV^{\eps}(t,x)|   \right\}
\leq \tilde L,\,\,\,\text{a.s.}
$$
with $\tilde L$ being independent of $\eps$ and $N$. Set $K=\tilde L$ and
\begin{align*}
\overline{V}^{\eps}(s,x)
&=
	V^{\eps}(s,x-\delta_N B_s)+Y^{\eps}_s+\delta_N \bar K y_t,\\
\underline{V}^{\eps}(s,x)
&=
	V^{\eps}(s,x-\delta_N B_s)-Y^{\eps}_s-\delta_N \bar K y_t,
\end{align*}
with $\bar K=4L(\tilde L+1)$ and $L$ the constant in $(\cA 1)$.

Notice that 
\begin{align*}
|\beta(t,x,v)-\beta(t,x-\delta_N B_t,v)|
+|f(t,x,v)-f(t,x-\delta_N B_t,v)| 
& \leq 2\delta_N L |B_t|,\\
|G(x)-G(x-\delta_NB_T)|
&\leq \delta_N L |B_T|.
\end{align*}
Then for $\overline V^{\eps}$ on $[t_{N-1},T)$, omitting the inputs for some involved functions, we have
\begin{align}
&-\mathfrak{d}_{t}\overline V^{\eps}-\bH(D\overline V^{\eps})
\nonumber\\
&=
	-\mathfrak{d}_{t}\overline V^{\eps}
	-\essinf_{v\in U} \bigg\{
  		     (\beta^N)'D\overline V^{\eps}
	  +f^N+f^{\eps}+ \tilde L  \beta^{\eps}\nonumber+\delta_N\bar K |B_t|\\
&
	  \quad\quad
	  +   \left(\beta-\beta^N\right)' D\overline V^{\eps}       -\beta^{\eps} \tilde L
	  +f-f^N-f^{\eps}-\delta_N \bar K |B_t|
	   \bigg\}
\nonumber\\
&\geq
	-\mathfrak{d}_{t}\overline V^{\eps}
	-\essinf_{v\in U} \bigg\{
  		      (\beta^N)'D\overline V^{\eps}
	  +f^N+f^{\eps}+\beta^{\eps} \tilde L 
	  +\delta_N \bar K |B_t|
	   \bigg\}
	 \label{est-A4}\\
&=0,\nonumber
\end{align}
and it follows similarly on intervals $[t_{N-2},t_{N-1})$, $\dots$, $[0,t_1)$ that
$$
-\mathfrak{d}_{t}\overline V^{\eps}-\bH(D\overline V^{\eps} ) \geq 0,
$$
which together with the obvious relation  $\overline V^{\eps}(T)=G^{\eps}+G^N+ \delta\bar K |B_T| \geq G$  indicates that $\overline V^{\eps}\in \overline {\mathscr V}$. Analogously, $\underline V^{\eps}\in \underline {\mathscr V}$.

Now let us measure the distance between $\underline V^{\eps}$, $\overline V^{\eps}$ and $V$. By the estimates for solutions of backward SDEs (see \cite[Proposition 3.2]{Hu_2002} for instance), we first have 
\begin{align*}
\|Y^{\eps}\|_{\cS^2(\bR)} + \|Z^{\eps}\|_{\cL^2(\bR^{m_0})}
&
	\leq  C \left(  \|G^{\eps}\|_{L^2(\Omega,\sF_T;\bR)} + \|f^{\eps}+\tilde L \beta^{\eps}\|_{\cL^2(\bR)}  \right)
\\
&
	\leq  C(1+\tilde L)\eps
\end{align*}
with the constant $C$ independent of $N$ and $\eps$.  Fix some $(s,x)\in[0,T)\times \bR^d$. In view of the approximation in Lemma \ref{lem-approx}, using It\^o's formula, Burkholder-Davis-Gundy's inequality, and Gronwall's inequality, we have through standard computations that for any $\theta\in\cU$,
\begin{align*}
&
	E_{\sF_s} \left [  \sup_{s\leq t\leq T} \left|  X^{s,x;\theta,N}_t-X^{s,x;\theta}_t   \right| ^2  \right]
\\
&\leq
	\tilde K\left(\delta_N ^2+ E_{\sF_s}\int_s^T\left| \beta^N\left(\tilde W_{t_1\wedge t},\cdots,\tilde W_{t_N\wedge t},t,X^{s,x;\theta,N}_t,\theta_t\right)-\beta\left(t,X^{s,x;\theta,N}_t,\theta_t\right) \right|^2\,dt\right)
	\\
&\leq
\tilde K \left( \delta_N^2+  E_{\sF_s}\int_s^T   \left|\beta^{\eps}_t\right|^2\,dt\right) ,
\end{align*}
with $\tilde K$ being independent of $s,\, x,\,N$, $\eps$ and $\theta$. Then
\begin{align*}
&E\left| V^{\eps}(s,x)-V(s,x)\right|
\\
&	
	\leq
		E\esssup_{\theta\in\cU}
		E_{\sF_s}\bigg[ 
		\int_s^T\Big( f^{\eps}_t +\Big| f\left(t,X^{s,x;\theta,N}_t,\theta_t\right) -f\left(t,X^{s,x;\theta}_t,\theta_t\right)\Big| \Big)\,dt
\\
&\quad\quad\quad
		+G^{\eps} +\Big| G\left(X^{s,x;\theta,N}_T\right) -G\left(X^{s,x;\theta}_T \right)\Big|
		\bigg] 
\\
&
	\leq E|Y^{\eps}_s|
		+
		2L(T^{1/2}+1)(\tilde K +1) \left( \delta_N + E \esssup_{\theta\in\cU}
		\left(
		E_{\sF_s}\int_s^T   \left|\beta^{\eps}_t\right|^2\,dt
		 \right)^{1/2}\right)
\\
&
	\leq \left\| Y^{\eps}\right\|_{\cS^2(\bR)} +2L(\tilde K+1) (T^{1/2}+1) \left( \delta_N + \left\|  \beta^{\eps} \right\|_{\cL^2(\bR^d)} \right)
\\
&
	\leq K_0 (\eps+\delta_N),
\end{align*}
with the constant $K_0$ being independent of $N$, $\eps$ and $(s,x)$. Furthermore, in view of the definitions of $\overline V^{\eps}$ and $\underline V^{\eps}$, there exists some constant $K_1$ independent of $\eps$ and $N$ such that
\begin{align*}
E\left| \overline V^{\eps}(s,x)-V(s,x)\right|
+E\left| \underline V^{\eps}(s,x)-V(s,x)\right|
\leq K_1 (\eps+\delta_N), \quad \forall \, (s,x)\in [0,T]\times\bR^d.
\end{align*}
The arbitrariness of $(\eps,\delta_N)$ together with the relation $\overline V^{\eps}\geq V \geq \underline V^{\eps}$ finally implies that $\underline u=V= \overline u$. 
\end{proof}
\begin{rmk}\label{rmk-superparab}
   In the above proof, by enlarging the original filtered probability space with an independent Brownian motion $B$,  we have actually constructed the regular approximations of $V$ with a regular perturbation induced by $\eps  B$, which does not necessitate the superparabolicity assumed in \cite{qiu2017viscosity}. This method may help to address the uniqueness of viscosity solution for certain classes of fully nonlinear degenerate stochastic HJB equations. Nevertheless, we would not do such a generalization in order to focus on the study of stochastic HJ equations in this work.
\end{rmk}


%
\bibliographystyle{acm}

\begin{thebibliography}{10}

\bibitem{Hu_2002}
{\sc Briand, P., Delyon, B., Hu, Y., Pardoux, E., and Stoica, L.}
\newblock $\textrm{L}^p$ solutions of backward stochastic differential
  equations.
\newblock {\em Stoch. Process. Appl. 108}, 4 (2003), 604--618.

\bibitem{buckdahn2015pathwise}
{\sc Buckdahn, R., Keller, C., Ma, J., and Zhang, J.}
\newblock Pathwise viscosity solutions of stochastic {PDEs} and forward
  path-dependent {PDEs}---a rough path view.
\newblock {\em arXiv:1501.06978\/} (2015).

\bibitem{buckdahn2007pathwise}
{\sc Buckdahn, R., and Ma, J.}
\newblock Pathwise stochastic control problems and stochastic {HJB} equations.
\newblock {\em SIAM J. Control Optim. 45}, 6 (2007), 2224--2256.

\bibitem{cont2013-founctional-aop}
{\sc Cont, R., Fourni{\'e}, D.-A., et~al.}
\newblock Functional it{\^o} calculus and stochastic integral representation of
  martingales.
\newblock {\em The Annals of Probability 41}, 1 (2013), 109--133.

\bibitem{crandall1992user}
{\sc Crandall, M.~G., Ishii, H., and Lions, P.-L.}
\newblock User's guide to viscosity solutions of second order partial
  differential equations.
\newblock {\em Bull. Amer. Math. Soc. 27}, 1 (1992), 1--67.

\bibitem{crandall2000lp}
{\sc Crandall, M.~G., Kocan, M., and {\'S}wiech, A.}
\newblock L$^p$-theory for fully nonlinear uniformly parabolic equations.
\newblock {\em Commun. Partial Differ. Equ. 25}, 11-12 (2000), 1997--2053.

\bibitem{da2014stochastic}
{\sc Da~Prato, G., and Zabczyk, J.}
\newblock {\em Stochastic equations in infinite dimensions}.
\newblock Cambridge university press, 2014.

\bibitem{DuQiuTang10}
{\sc Du, K., Qiu, J., and Tang, S.}
\newblock $\textrm{L}^p$ theory for super-parabolic backward stochastic partial
  differential equations in the whole space.
\newblock {\em Appl. Math. Optim. 65}, 2 (2011), 175--219.

\bibitem{ekren2014viscosity}
{\sc Ekren, I., Keller, C., Touzi, N., and Zhang, J.}
\newblock On viscosity solutions of path dependent {PDEs}.
\newblock {\em Ann. Probab. 42}, 1 (2014), 204--236.

\bibitem{ekren2016viscosity-1}
{\sc Ekren, I., Touzi, N., and Zhang, J.}
\newblock Viscosity solutions of fully nonlinear parabolic path dependent
  {PDEs: Part I}.
\newblock {\em Ann. Probab. 44}, 2 (2016), 1212--1253.

\bibitem{gawarecki2010stochastic}
{\sc Gawarecki, L., and Mandrekar, V.}
\newblock {\em Stochastic differential equations in infinite dimensions: with
  applications to stochastic partial differential equations}.
\newblock Springer Science \& Business Media, 2010.

\bibitem{Horst-Qiu-Zhang-14}
{\sc Horst, U., Qiu, J., and Zhang, Q.}
\newblock A constrained control problem with degenerate coefficients and
  degenerate backward {SPDEs} with singular terminal condition.
\newblock {\em SIAM J. Control Optim. 54}, 2 (2016), 946--963.

\bibitem{Hu_Ma_Yong02}
{\sc Hu, Y., Ma, J., and Yong, J.}
\newblock On semi-linear degenerate backward stochastic partial differential
  equations.
\newblock {\em Probab. Theory Relat. Fields 123\/} (2002), 381--411.

\bibitem{juutinen2001definition}
{\sc Juutinen, P.}
\newblock On the definition of viscosity solutions for parabolic equations.
\newblock {\em Proceedings of the American Mathematical Society 129}, 10
  (2001), 2907--2911.

\bibitem{krylov1982boundedly}
{\sc Krylov, N.~V.}
\newblock Boundedly nonhomogeneous elliptic and parabolic equations.
\newblock {\em Izvestiya Rossiiskoi Akademii Nauk. Seriya Matematicheskaya 46},
  3 (1982), 487--523.

\bibitem{Krylov-1987}
{\sc Krylov, N.~V.}
\newblock {\em Nonlinear Elliptic and Parabolic Equations of the Second Order}.
\newblock D. Reidel, Dordrecht, 1987.

\bibitem{Leao-etal-2018}
{\sc Le$\tilde{a}$o, D., Ohashi, A., and Simas, A.}
\newblock A weak version of path-dependent functional {It\^o} calculus.
\newblock {\em Ann. Probab. 46(6), 2018, 3399--3441\/} .

\bibitem{LionsSouganidis1998b}
{\sc Lions, P., and Souganidis, P.}
\newblock Fully nonlinear stochastic partial differential equations: Non-smooth
  equations and applications.
\newblock {\em C.R. Acad. Sci. paris 327}, 1 (1998), 735--741.

\bibitem{lions-1983}
{\sc Lions, P.~L.}
\newblock Optimal control of diffusion processes and {Hamilton-Jacobi-Bellman}
  equations, {Part II}.
\newblock {\em Commun. Partial Differ. Equ. 8\/} (1983), 1229--1276.

\bibitem{lukoyanov2007viscosity}
{\sc Lukoyanov, N.~Y.}
\newblock On viscosity solution of functional hamilton-jacobi type equations
  for hereditary systems.
\newblock {\em Proceedings of the Steklov Institute of Mathematics 259}, 2
  (2007), S190--S200.

\bibitem{ma2012non}
{\sc Ma, J., Yin, H., and Zhang, J.}
\newblock On non-{M}arkovian forward--backward {SDEs} and backward stochastic
  {PDEs}.
\newblock {\em Stoch. Process. Appl. 122}, 12 (2012), 3980--4004.

\bibitem{oksendal2003stochastic}
{\sc {\O}ksendal, B.}
\newblock {\em Stochastic differential equations}.
\newblock Springer, 2003.

\bibitem{Pardoux1979}
{\sc Pardoux, E.}
\newblock Stochastic partial differential equations and filtering of diffusion
  processes.
\newblock {\em Stoch.\/} (1979), 127--167.

\bibitem{Peng_92}
{\sc Peng, S.}
\newblock Stochastic {H}amilton-{J}acobi-{B}ellman equations.
\newblock {\em {SIAM} J. Control Optim. 30\/} (1992), 284--304.

\bibitem{peng2011backward}
{\sc Peng, S.}
\newblock Backward stochastic differential equation, nonlinear expectation and
  their applications.
\newblock In {\em Proceedings of the International Congress of
  Mathematicians\/} (2010), pp.~393--432.

\bibitem{peng2011note}
{\sc Peng, S.}
\newblock Note on viscosity solution of path-dependent {PDE and G-martingales}.
\newblock {\em arXiv:1106.1144\/} (2011).

\bibitem{Qiu2014weak}
{\sc Qiu, J.}
\newblock Weak solution for a class of fully nonlinear stochastic
  hamilton--jacobi--bellman equations.
\newblock {\em Stoch. Process. Appl. 127}, 6 (2017), 1926--1959.

\bibitem{qiu2017viscosity}
{\sc Qiu, J.}
\newblock Viscosity solutions of stochastic {Hamilton--Jacobi--Bellman}
  equations.
\newblock {\em {SIAM} J. Control Optim. 56}, 5 (2018), 3708--3730.

\bibitem{QiuTangMPBSPDE11}
{\sc Qiu, J., and Tang, S.}
\newblock Maximum principles for backward stochastic partial differential
  equations.
\newblock {\em J. Funct. Anal. 262\/} (2012), 2436--2480.

\bibitem{Tang_05}
{\sc Tang, S.}
\newblock Semi-linear systems of backward stochastic partial differential
  equations in $\mathbb{R}^n$.
\newblock {\em Chin. Ann. Math. 26B}, 3 (2005), 437--456.

\bibitem{Tang-Wei-2013}
{\sc Tang, S., and Wei, W.}
\newblock On the cauchy problem for backward stochastic partial differential
  equations in {H{\"o}lder} spaces.
\newblock {\em Ann. Probab. 44}, 1 (2016), 360--398.

\bibitem{wang1992-I}
{\sc Wang, L.}
\newblock On the regularity theory of fully nonlinear parabolic equations: {I}.
\newblock {\em Commun. Pure Appl. Math. 45}, 1 (1992), 27--76.

\end{thebibliography}

\end{document}